\documentclass[11 pt,final]{article}

\usepackage{amsmath,amssymb,amsthm,amscd,amsfonts}
\usepackage{epsfig,pinlabel}
\usepackage[hmargin=1.25in, vmargin=1in]{geometry}
\usepackage[font=small,format=plain,labelfont=bf,up,textfont=it,up]{caption}
\usepackage{subcaption}
\usepackage{stmaryrd}
\usepackage{verbatim}
\usepackage{type1cm}
\usepackage{url}
\usepackage{calc}
\usepackage[all,cmtip]{xy}
\usepackage{graphicx}
\usepackage{multicol}
\usepackage{pstricks}
\usepackage{relsize}
\usepackage{float}
\usepackage{overpic}
\usepackage{enumerate}

\newcommand{\urlwofont}[1]{\urlstyle{same}\url{#1}}

\newcommand{\nc}{\newcommand}
\nc{\nt}{\newtheorem}
\nc{\dmo}{\DeclareMathOperator}

\theoremstyle{plain}
\newtheorem{theorem}{Theorem}[section]
\newtheorem{maintheorem}{Theorem}
\newtheorem{prop}[theorem]{Proposition}

\newtheorem{corollary}[theorem]{Corollary}
\newtheorem{maincorollary}[maintheorem]{Corollary}

\theoremstyle{definition}

\theoremstyle{remark}

\dmo{\SMod}{SMod}
\dmo{\PMod}{PMod}
\dmo{\SHomeo}{SHomeo}
\dmo{\SI}{\mathcal{SI}}
\dmo{\SSp}{SSp}
\dmo{\PSp}{PSp}

\DeclareMathOperator{\GL}{GL}

\newcommand\Z{\ensuremath{\mathbb{Z}}}
\newcommand\Q{\ensuremath{\mathbb{Q}}}

\nc{\p}[1]{\noindent {\bf #1.}}
\nc{\margin}[1]{\marginpar{\scriptsize #1}}

\nc{\PartialIBases}{\mathfrak{IB}}
\nc{\PartialIBasesEx}{\widehat{\mathfrak{IB}}}
\nc{\PartialBases}{\mathfrak{B}}
\nc{\Building}{\mathfrak{T}}
\nc{\height}{\ensuremath{\text{ht}}}
\nc{\Poset}{\mathfrak{P}}
\nc{\Field}{\mathbb{F}}
\nc{\Link}{\ensuremath{\text{Link}}}
\nc{\Star}{\ensuremath{\text{Star}}}
\nc{\SymTorelli}{\ensuremath{\mathcal{SI}}}
\nc{\BTorelli}{\ensuremath{\mathcal{BI}}}
\dmo{\Braid}{\ensuremath{B}}
\dmo{\PureBraid}{\ensuremath{PB}}
\nc{\Hyper}{\ensuremath{\iota}}

\nc{\BigFreeProd}{\mathop{\mbox{\Huge{$\ast$}}}}

\nc{\Quotient}{\ensuremath{\mathcal{Q}}}
\nc{\QuotientEx}{\ensuremath{\widehat{\mathcal{Q}}}}

\nc{\Presentation}[2]{\ensuremath{\text{$\langle #1$ $|$ $#2 \rangle$}}}
\nc{\SpGen}{\ensuremath{S_{\text{Sp}}}}
\nc{\SpRel}{\ensuremath{R_{\text{Sp}}}}
\nc{\QGen}{\ensuremath{S_{\mathcal{Q}}}}
\nc{\QRel}{\ensuremath{R_{\mathcal{Q}}}}
\nc{\PBs}{\ensuremath{T}}
\nc{\Qs}{\ensuremath{\overline{s}}}

\dmo{\PB}{PB}
\nc{\BIredg}{\mathcal{BI}_{2g+1}^{\text{red}}}
\nc{\BI}{\mathcal{BI}}
\dmo{\D}{D}
\dmo{\Stab}{Stab}
\dmo{\Surger}{Surger}
\nc{\I}{\mathcal{I}}

\nc{\spanmap}{span}

\nc{\genbygen}[2]{\premonoid{#1}{#2}}
\nc{\premonoid}[2]{#1 \circledcirc #2}
\nc{\monoid}[2]{#1 \odot #2}

\nc{\G}{\Gamma}
\nc{\raag}{A_\Gamma}
\nc{\racg}{W_\G}
\nc{\raagdelt}{A_\Delta}
\dmo{\Aut}{Aut}
\dmo{\Out}{Out}
\nc{\Autraag}{\Aut(\raag)}
\nc{\Outraag}{\Out(\raag)}
\nc{\diag}{D_\G}
\nc{\Autraagdelt}{\Aut(A_\Delta)}
\nc{\glk}{\GL(k,\mathbb{Z})}
\nc{\gln}{\GL(n,\mathbb{Z})}
\nc{\glnt}{\GL(n,\mathbb{Z} / 2)}
\nc{\glkdelt}{\GL(k |\Delta| ,\mathbb{Z})}
\nc{\gldelt}{\GL(|\Delta| ,\mathbb{Z})}
\nc{\zkdelt}{\mathbb{Z}^{k|\Delta|}}
\nc{\join}{\mathcal{J}}
\nc{\pc}{\mathrm{PC}}
\dmo{\lk}{lk}
\dmo{\st}{st}
\dmo{\Inn}{Inn}

\nc{\pia}{\Pi \mathrm{A}}
\nc{\piaG}{\pia_\G}
\nc{\ppia}{\mathrm{P \Pi A}}
\nc{\ppiaG}{\mathrm{P \Pi A}_\G}
\nc{\epia}{\mathrm{E \Pi A}}
\nc{\epiaG}{\mathrm{E \Pi A}_\G}
\nc{\ptor}{\mathcal{PI}}
\nc{\ptorG}{\mathcal{PI}_\G}
\nc{\CGi}{C_\G(\iota)}
\nc{\pbc}{\mathfrak{B}^\pi}
\dmo{\Cay}{Cay}
\dmo{\rev}{rev}
\nc{\Autfn}{\Aut(F_n)}
\dmo{\supp}{supp}
\dmo{\rk}{\mathrm{rk}}
\dmo{\PCT}{\mathrm{PCT}(\raag)}
\dmo{\PCTo}{\overline{\mathrm{PCT}}(\raag)}

\include{diagram}
\title{\vspace{-30pt} Infinite groups acting faithfully on the outer automorphism group of a right-angled Artin group }

\author{Corey Bregman and Neil J. Fullarton\vspace{-6pt}}

\begin{document}

\newcounter{enumi_saved}

\maketitle

\vspace{-23pt}
\begin{abstract}We construct the first known examples of infinite subgroups of the outer automorphism group of $\Outraag$, for certain right-angled Artin groups $\raag$. This is achieved by introducing a new class of graphs, called \emph{focused graphs}, whose properties allow us to exhibit (infinite) projective linear groups as subgroups of $\Out(\Outraag)$. This demonstrates a marked departure from the known behavior of $\Out(\Outraag)$ when $\raag$ is free or free abelian, as in these cases $\Out(\Outraag)$ has order at most 4. We also disprove a previous conjecture of the second author, producing new examples of finite order members of certain $\Out(\Autraag)$.\end{abstract}

%\margin{\textcolor{blue}{Corey's comments in blue.} \textcolor{red}{Neil's comments in red.}}

\section{Introduction}

Right-angled Artin groups, or \emph{RAAGs}, comprise a class of groups which generalize free groups and free abelian groups. Every finite simplicial graph $\Gamma$ with vertex set $V$ defines a RAAG $A_\Gamma$ in the following way. The generating set of $A_\G$ is in bijection with the vertices of $\Gamma$ and the only relations are that two generators commute if their corresponding vertices share an edge in $\Gamma$.  Thus if $\Gamma$ has no edges then $A_\Gamma$ is just the free group $F_{V}$, whereas if $\Gamma$ is a complete graph, $A_\Gamma$ is the free abelian group $\Z\langle V\rangle$.  

%The graphical representation of raags conjures up a certain mental image, namely that one can actually \emph{interpolate} between free non-abelian and free abelian groups via raags.  Start with the totally disconnected graph on the vertex set $V$; this represents $F_V$.  Add in edges one at a time until one obtains the complete graph on $V$; this corresponds to $\Z\langle V\rangle$.  The sequence of raags along the way interpolate between $F_V$ and $\Z\langle V\rangle$.  One wonders how well this intrerpolation corresponds to reality.  More precisely, which group theoretic properties are shared by all the groups appearing in the sequence.  

%\margin{\textcolor{blue}{I rather like this story but probably we should have both of the main theorems on the first page...}}
In this paper, we will consider the automorphism and outer automorphism groups of general RAAGs in comparison with those of free groups and free abelian groups.  More specifically, we will investigate $\Out(\Out(A_\Gamma))$ and $\Out(\Aut(A_\Gamma))$. These groups provide a measure of the algebraic rigidity of $\Outraag$ and $\Autraag$, respectively, and their study fits into a more general program of investigating rigidity of groups throughout geometric group theory. 

The main goal of this paper is to show that there exist infinitely many graphs $\Gamma$ for which $\Out (\Outraag)$ is infinite. We achieve this by introducing a new class of graphs, which we call \emph{focused graphs}. A graph $\Gamma$ is said to be \emph{focused} if it has a distinguished vertex $c$ with the following two properties: (i) $c$ is the unique vertex of $\Gamma$ that may dominate a vertex other than itself, and (ii) $c$ is the only vertex whose star disconnects $\Gamma$. Focused graphs are the key construction that allow us to prove our following main theorem.

\begin{maintheorem} \label{inf} For each $n\geq 2$, there exist infinitely many focused graphs $\Gamma$ such that $\Out(\Out(A_\Gamma))$ contains $\mathrm{PGL}_{n}(\Z)$. \end{maintheorem}

Previous work has calculated the groups $\Out(\Autraag)$ and $\Out(\Autraag)$ when $\raag$ is a free or free abelian group. A classical result of Hua and Reiner \cite{HR51} computes \linebreak $\Out(\Aut(\Z^n))=\Out(\Out(\Z^n))=\Out(\mathrm{GL}_{n}(\Z))$ to be \[\Out(\mathrm{GL}_{n}(\Z))=\left\{\begin{array}{cl} \Z/2\times \Z/2, & \mbox{even } n,\\
\Z/2, & \mbox{odd } n>1,\\
1, &n=1.

\end{array}\right.
\]
For the case of free groups, Dyer and Formanek give an algebraic proof in \cite{DF75} that \linebreak $\Out(\Aut(F_n))$ is trivial for all $n$. In \cite{K90}, Khramtsov gave another proof of this fact, and also showed that $\Out(\Out(F_n))$ is trivial for $n\geq 3$. Using outer space and auter space for free groups, Bridson and Vogtmann \cite{BV00} gave a geometric proof that for $n\geq 3$ both $\Out(\Out(F_n))$ and $\Out(\Aut(F_n))$ are trivial.  Note that the cases $n=1$ and $n=2$ for $\Out(\Out(F_n))$ are covered by the Hua--Reiner theorem, since $F_1 \cong \Z$, and a theorem of Nielsen states that $\Out(F_2)\cong \mathrm{GL}_2(\Z)$ (see \cite{LS77}).  

The above results indicate that both $\Out(\Autraag)$ and $\Out(\Outraag)$ are either small or trivial for $A_\G=\Z^n$ and $A_\G=F_n$, independent of $n$. For more general RAAGs, the second author has shown in \cite{F14} that this behavior does not hold.   More precisely, he proves that for any $n>0$ there exist graphs $\G_1, \Gamma_2$ so that $|\Out(\Aut(A_{\G_1}))|>n$, and $|\Out(\Out(A_{\G_2}))|>n$.  Theorem~\ref{inf} of this paper substantially strengthens the second author's result in the case of $\Out(\Out(A_\G))$. 

%Our approach is to define an infinite family of graphs called \emph{focused graphs}, each of which yields a RAAG that satisfies the conclusion of Theorem~\ref{inf}.  For each focused graph $\G$ we compute a large subgroup of $\Out(\Out(A_\G))$ explicitly, showing in particular that for each $k \geq 1$, $\mathrm{PGL}_k(\Z)$ is contained in $\Out(\Out(A_\G))$ for some $\G$. In showing this, we first exhibit $\mathrm{GL}_{k}(\Z)$ subgroups inside $\Aut(\Outraag)$, hence proving that any $\Z$-linear group can be made to act faithfully on $\Outraag$ via automorphisms, for some RAAG $\raag$. This provides a stark constrast to the work of Hua--Reiner, Bridson--Vogtmann and Dyer--Formanek summarized above. 

Our approach is to compute explicitly a large subgroup of $\Out(\Out(A_\G))$  for each focused graph $\Gamma$. In computing this, we first exhibit $\mathrm{GL}_{n}(\Z)$ subgroups inside $\Aut(\Outraag)$, hence proving that any $\Z$-linear group can be made to act faithfully on $\Outraag$ via automorphisms, for some RAAG $\raag$. This provides a stark constrast to the work of Hua--Reiner, Bridson--Vogtmann and Dyer--Formanek summarized above. 
 
The second author \cite{F14} also introduced the notion of an \emph{austere graph}.  If $\Gamma$ is austere, then $\Out(A_\Gamma)$ is, in some sense, as simple as possible.  The second author previously conjectured that for austere graphs $\G$, the group $\Autraag$ is complete (see the remarks after Proposition 5.1 in \cite{F14}). However the following theorem establishes that the order of $\Out(\Autraag)$ in the austere case is at least exponential in $n$.

\begin{maintheorem}\label{aust} If $\Gamma$ is austere and $n = |V| > 1$, then $|\Out(\Aut(A_\Gamma))|\geq2^n$.  
\end{maintheorem}
In particular, we are able to achieve the two main results of \cite{F14} simultaneously:
\begin{maincorollary} For each $n\geq1$, there exist infinitely many graphs $\Gamma$ such that \linebreak $|\Out(\Aut(A_\Gamma))|>n$ and $|\Out(\Out(A_\Gamma))|>n$.
\end{maincorollary}

\textbf{A caveat.} One might na\"{i}vely expect that in order to construct automorphisms of $\Outraag$, say, it would suffice to find a finite index subgroup $K\leq \Outraag$ that has a rich collection of automorphisms as an abstract group. It could then be hoped that these extend to give many automorphisms of $\Outraag$, since it is often the case that group-theoretic properties pass easily between a group and its finite index subgroups. Indeed, this is our approach, however the interplay between the finitely many cosets of $K$ in $\Outraag$ frequently prohibits any obvious attempts at extending such automorphisms to all of $\Outraag$. 

Considering the abstract commensurator $\text{Comm}(\Outraag)$ instead of $\Out(\Outraag)$ circumvents some of these difficulties, since $\text{Comm}(\Outraag)$ is precisely concerned with isomorphisms between finite index subgroups of $\Outraag$. For details see the remarks after Corollary \ref{faith}.

%For instance, in Section~\ref{austere}, we consider groups $\Autraag$ that have a finite index subgroup isomorphic to $\raag$ itself, but only finitely many automorphisms of this subgroup extend to all of $\Autraag$ in an obvious fashion (and in fact, these obvious extensions lie in $\Inn(\Autraag)$!). To find non-trivial members of $\Out(\Autraag)$, we look elsewhere, constructing automorphisms that do not preserve the finite index subgroup $\raag$.

\textbf{Outline of the paper.} In Section 2, we recall some necessary background regarding automorphisms of right-angled Artin groups. In Section 3, we prove Theorem~\ref{inf}, while in Section 4, we prove Theorem~\ref{aust}.

\textbf{Acknowledgements.} The authors are grateful to Dan Margalit and Benson Farb for their helpful remarks on a draft of this paper.  The authors would also like to thank Tara Brendle and Andrew Putman for many useful comments and discussions, and for their persistent encouragement.  

\section{Preliminaries}
In this section we review basic properties of RAAGs and their automorphism groups.  Let $\G = (V,E)$ be a simplicial graph.  As stated in the introduction, $\G$ defines a group $A_{\G}$ with generating set $V=\{v_1,\ldots,v_n\}$ and relations $v_i v_j=v_j v_i$ if and only if $v_i$ is adjacent to $v_j$ in $\G$. For $v\in V$, denote by $\lk(v)$ the \emph{link} of $v$, by which we mean the set of vertices adjacent to $v$. The \emph{star} of $v$, by which we mean the set $\lk(v) \cup \{ v \}$, will be denoted $\st(v)$.  If $u,v\in V$ and $\lk(v)\subseteq \st(u)$ then we say $u$ \emph{dominates} $v$ and write $v\leq u$.  

%Due to the close relationship between raags and Coxeter groups \margin{\textcolor{red}{Is this true? Surely the `reason' is the similarity with free and free abelian groups}}, 

Elements of $A_\G$ enjoy nice normal forms in terms of the generators $V$. Two words $w_1$ and $w_2$ in the generators $V$ (and their inverses) are said to be \emph{shuffle-equivalent} if $w_2$ can be obtained from $w_1$ by repeatedly exchanging pairs of adjacent commuting generators.  Hermiller and Meier show in \cite{HM99} that if $w_1$ and $w_2$ are minimal length words, then $w_1=w_2$ in $A_\G$ iff $w_1$ is shuffle-equivalent to $w_2$, and moreover that any word can be transformed into a minimal length word by swapping adjacent commuting generators and cancelling pairs of inverses whenever possible. This allows us to define the \emph{support} of $w \in \raag$, denoted $\supp(w)$, to consist of all $v \in V$ such that $v$ (or $v^{-1}$) appears in a minimal length word representing $w$. For a survey of RAAGs and their properties, see \cite{Ch07}.  
  
The automorphism group $\Autraag$ of a RAAG $\raag$ is generated by the following four types of automorphisms, known as the Laurence--Servatius generators:
\begin{enumerate}
\item \emph{Inversions}: Given $v\in V$, the automorphism $\iota_v$ sends $v\mapsto v^{-1}$ and fixes all other generators.  
\item \emph{Graph automorphisms}: Any graph automorphism of $\G$ induces a permutation of $V$ which extends to an automorphism of $A_\G$.
\item \emph{Transvections}: If $v\leq u$, the automorphism $\tau_{uv}$ sends $v\mapsto uv$ and fixes all other generators. 
\item \emph{Partial conjugations}: If $P$ is a connected component of $\G\setminus \st(v)$ for some $v\in V$, the automorphism $\chi_{v,P}$ maps $u\mapsto vuv^{-1}$ for every $u\in P$, and acts as the identity elsewhere.  
\end{enumerate}
The fact that these four types of automorphisms generate $\Aut(A_\G)$ was conjectured by Servatius in \cite{Ser89}, and later proven by Laurence \cite{Lau95}.  If $v\leq u$ and $v$ is adjacent to $u$, then $\tau_{uv}$ is an \emph{adjacent} transvection.  Otherwise $\tau_{uv}$ is a \emph{non-adjacent} transvection. In the sequel, the subgroup generated by the inversions will be denoted $I_\G$, while the subgroup generated by partial conjugations and tranvections will be denoted $\PCT$. The images of these four types of generators under the quotient map $\Aut(A_\G)\rightarrow \Out(A_\G)$ generate $\Out(A_\G)$.  We will use an overline to indicate when we refer to elements or subgroups of $\Out(A_\G)$: for example, $\bar \tau_{uv}$, $\bar\chi_{v,P}$, and $\PCTo$. 

%Basic properties of raags (in particular, shuffle-equivalence, since we use it later)
%
%Laurence-Servatius generators

\section{Proof of Theorem A}

Let $\G = (V,E)$ be a graph with a distinguished vertex $c$, such that if $v \leq u$ for distinct $v,u \in V$, then $u = c$, and for any $v \in V \setminus \{c\}$, the graph $\Gamma \setminus \st(v)$ is connected. We will call such a graph \emph{focused} (\emph{at $c$}). Let $L~=~\{ x_1, \ldots, x_l \} \subset V \setminus \{c\}$ denote the set of vertices that are dominated by, but not adjacent to, the vertex $c$, and let $S~=~\{ x_{l+1}, \ldots, x_m\} \subset V \setminus \{c\}$ denote the set of vertices that are both dominated by and adjacent to $c$. Finally, let $Q~=~\{ P_1, \ldots , P_{k} \}$ denote the connected components of the graph $\Gamma \setminus \st(c)$, where $k \geq l$, and we set $P_i = \{ x_i \}$ for $1 \leq i \leq l$. See Figure \ref{f1} below for a typical example of a focused graph. 

\begin{figure}[h]
\centering
\includegraphics[width=2.5in]{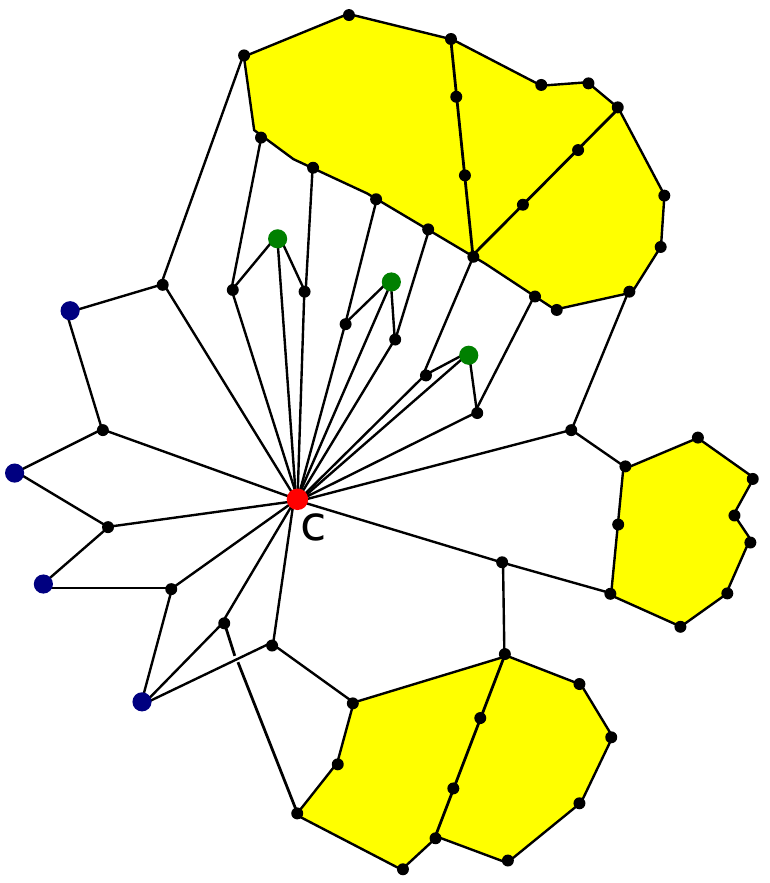}
\caption{An example of a focused graph with $l=4$, $m=7$, and $k=7$.  The distinguished vertex $c$ is shown in red.  Vertices dominated by but not adjacent to $c$ are shown in blue, while those which are dominated by and adjacent to $c$ are green. Connected components of $\G\setminus \st(c)$ which are not vertices are yellow.}
\label{f1}
\end{figure}

Note that a focused graph $\G$ may have non-trivial graph automorphism group, $\Aut (\G)$. From now on, we assume that $\Aut(\G)$ is trivial, as this simplifies the following exposition. This is not too restrictive a condition; our construction still yields infinite subgroups of $\Out(\Outraag)$ if $\Aut(\G) \neq 1$, however the obvious action of $\Aut(\G)$ on $\Outraag$ would force us to pass to proper subgroups of those we find below.

Examining the Laurence--Servatius generators of $\Autraag$ for a focused graph $\G$, we find that \[ \Autraag \cong \PCT \rtimes I_\G, \] with the splitting following from the observation that $I_\G$ injects into $\mathrm{GL}_{n}(\Z)$ under the canonical map $\Phi: \Autraag \to \mathrm{GL}_{n}(\Z)$, and $\Phi(\PCT) \cap \Phi(I_\G) = 1$. This splitting of $\Autraag$ descends to one of $\Outraag$, since $\Inn(\raag) \leq \PCT$. The image $\PCTo$ of $\PCT$ in $\Outraag$ is easy to describe, which we do now explicitly.

\begin{prop}\label{pct}Let $\Gamma$ be focused at $c$. Then \[ \Outraag \cong \Z^{k+m-1} \rtimes I_\G .\] In particular, $\PCTo \cong \Z^{k+m-1}$.
\end{prop}
\begin{proof}The group $\Outraag$ certainly splits as $\PCTo \rtimes I_\G$, by the preceding discussion. We must show that $\PCTo$ is free abelian of rank $k+m-1$. 

The group $\PCT \leq \Autraag$ is generated by $\Inn(\raag) \cong \raag$ and the $m +k$ mutually-commuting Laurence--Servatius generators of the form $\tau_{c{x_i}}$ or $\chi_{c,P_j}$, for $1 \leq i \leq m$ and $1 \leq j \leq k$. For conciseness, we shall write $\tau_i = \tau_{c{x_i}}$ and $\chi_j = \chi_{c,P_j}$.

The image $\PCTo$ is, therefore, an abelian group generated by the images $\bar \tau_i$ and $\bar \chi_j$ of $\tau_i$ and $\chi_j$ (respectively), and we use additive notation to reflect this. Since \[\gamma_c = \prod_{j=1}^{k} \chi_j\] in $\Autraag$, where $\gamma_c \in \Inn(\raag)$ denotes conjugation by $c$, we observe that \[\bar \chi_1 + \dots + \bar \chi_{k-1} = - \bar  \chi_{k}\] in $\Outraag$. We thus remove $\bar \chi_k$ from our generating set for $\Outraag$.

%We now show that the product $\omega := {\tau_1}^{r_1} \dots {\tau_m}^{r_m} \chi_1^{s_1} \dots {\chi_{k-1}}^{s_{k-1}} \in \Autraag$ ($r_i$, $s_j \in \Z$) is inner if and only if it is trivial. Suppose $\omega$ is equal to conjugation by $p \in \raag$.  Since $\omega$ acts trivially on $\st(c)$, we must have $pvp^{-1} = v$ for each $v \in \st(c)$. This implies that every $u \in \supp(p)$ is adjacent to each such $v$, since $pvp^{-1}$ and $v$ must be shuffle-equivalent.  Each vertex in $\supp(p)$ hence dominates the vertex $c$, and so $\supp(p) = \{c\}$ or $\emptyset$, since $\Gamma$ is focused at $c$. 

We now show that the product $\omega := {\tau_1}^{r_1} \dots {\tau_m}^{r_m} \chi_1^{s_1} \dots {\chi_{k-1}}^{s_{k-1}} \in \Autraag$ ($r_i$, $s_j \in \Z$) is inner if and only if it is trivial. First, observe that each of the $r_i$ must be zero. For if $[\omega]$ denotes the induced action of $\omega$ on the abelianization of $A_\G$, we have $[\omega]\colon [x_i]\mapsto [x_i]+r_i[c]$, where $[x_i]$ and $[c]$ denote the equivalence classes of $x_i$ and $c$ in the abelianization of $A_\G$.  

Hence we may assume $\omega:=\chi_1^{s_1} \dots {\chi_{k-1}}^{s_{k-1}}$ and suppose $\omega$ is equal to conjugation by $p \in \raag$. Since $\omega$ acts trivially on $\st(c)$, we must have $pvp^{-1} = v$ for each $v \in \st(c)$. This implies that every $u \in \supp(p)$ is adjacent to each such $v$, since $pvp^{-1}$ and $v$ must be shuffle-equivalent.  Each vertex in $\supp(p)$ hence dominates the vertex $c$, and so $\supp(p) = \{c\}$ or $\emptyset$, since $\Gamma$ is focused at $c$. 

However, $\omega$ also acts trivially on $P_k$, so by the same argument, we must have that $p$ is the identity, since $c$ is not adjacent to any vertex in $P_k$. Thus, if $\omega$ is non-trivial in $\Autraag$, its image is non-trivial in $\Outraag$, and so the set $\{\bar \tau_i, \bar \chi_j \mid 1 \leq i \leq m, 1 \leq j \leq k-1 \}$ is a free abelian basis for $\PCTo$.
\end{proof}

Since the image of $\PCT$ in $\Outraag$ is torsion-free, and so is $\Inn(\raag)$, we obtain the following corollary.
\begin{corollary} For a focused graph $\G$, the group $\PCT$ is torsion-free.
\end{corollary}

Our goal is now to understand the action of $I_\G$ on $\PCTo \cong \Z^{k+m-1}$ sufficiently to identify automorphisms of $\Z^{k+m-1}$ that may extend to well-defined automorphisms of $\Outraag$ by declaring that they act trivially on $I_\G$.

Due to its distinguished role, we denote by $\iota_c$ the automorphism of $\raag$ that inverts $c \in V$ and fixes every $v \in V \setminus \{c \}$. The action of $I_\G$ is fully encoded by the following six types of relation:
\begin{align} \iota_c \bar \chi_j \iota_c &=- \bar \chi_j & (1 \leq j \leq k-1 ), \\
\iota_c \bar\tau_i \iota_c &= -\bar\tau_i & (1 \leq i \leq m), \\
\iota_r \bar\chi_j \iota_r &= \bar\chi_j & (1 \leq j \leq k-1, \iota_r \neq \iota_c), \\
\iota_r \bar\tau_i \iota_r &= \bar\tau_i & (1 \leq i \leq m, \iota_r \neq \iota_c \mbox{ or } \iota_i),\\
\iota_i \bar\tau_i \iota_i &= \bar\chi_i -\bar\tau_i & (1 \leq i \leq l), \\
\iota_i \bar\tau_i \iota_i &= -\bar\tau_i & (l+1 \leq i \leq m).
\end{align}

Note that relations (5) and (6) distinguish $\iota_i \bar\tau_i \iota_i$ depending upon whether $\bar\tau_i$ is a non-adjacent or adjacent transvection, respectively. The action of $I_\Gamma$ on $\PCTo$ in the semi-direct product decomposition of $\Outraag$ is given by a homomorphism \[ \alpha: I_\G \to \Aut (\Z^{k+m-1}) \cong \mathrm{GL}_{k+m-1}(\Z). \] Let $\mathcal{C}$ denote the centralizer of $\alpha(I_\G)$ in $\mathrm{GL}_{k+m-1}(\Z)$. We may view $\mathcal{C}$ as a subgroup of $\Aut(\Outraag)$ by extending each $M \in \mathcal{C}$ to an automorphism $\tilde M \in \Aut(\Outraag)$ by declaring that $\tilde M$ restricts to the identity on $I_\G$ (see \cite[Section 3.1]{F14} for a more detailed discussion).

In order to give a tractable description of $\alpha(I_\G)$ and $\mathcal{C}$, we order the free basis for $\PCTo$ found in Proposition~\ref{pct} as follows:
\[( \bar \chi_1, \bar \tau_1, \dots, \bar \chi_l, \bar \tau_l, \bar \tau_{l+1}, \dots, \bar \tau_m, \bar \chi_{l+1}, \dots, \bar \chi_{k-1} ).\] As is usual, we denote the $q \times q$ identity matrix by $I_q$. Looking at relations (1)--(6) above, we see that the subgroup $\alpha(I_\G)$ consists of $-I_{k+m-1}$ together with block-diagonal matrices of the form $\mathrm{Diag}(D_1,D_2,D_3)$, where $D_3$ is $\pm I_{k-l-1}$ and $D_2$ is any diagonal matrix in $\mathrm{GL}_{m-l}(\Z)$. The matrix $D_1$ is any matrix in $\mathrm{GL}_{2l}(\Z)$ with block decomposition
\[ \begin{pmatrix} 
A_1 & 0 & \cdots & 0 \\ 0 & A_2 & \cdots &  0 \\
\vdots & \vdots & \ddots & \vdots \\
0 & 0 & \cdots & A_l 
\end{pmatrix}, \]

where each $A_i$ ($1 \leq i \leq l$) is either $I_2$ or $\begin{pmatrix} 1 & 1 \\ 0 & -1 \end{pmatrix}$. We denote the subgroup of $\mathrm{GL}_{2l}(\Z)$ consisting of such matrices by $\mathcal{L}$. 

With this description of $\alpha(I_\G)$ in place, we now identify the centralizer $\mathcal{C}$. We denote by $\Lambda_l[2]$ the \emph{principal level 2 congruence subgroup of $\mathrm{GL}_l(\Z)$} (that is, the kernel of the epimorphism $\mathrm{GL}_l(\Z) \to \mathrm{GL}_l(\Z / 2)$ that reduces matrix entries mod 2).

\begin{prop}\label{cent} The centralizer $\mathcal{C}$ of $\alpha(I_\G)$ in $\mathrm{GL}_{k+m-1}(\Z)$ is isomorphic to \[ \Lambda_l[2] \times (\Z / 2)^{m} \times \mathrm{GL}_{k-1}(\Z). \]
\end{prop}

\begin{proof} Let $M \in \mathrm{GL}_{k+m-1}(\Z)$, and suppose that $M$ centralizes $\alpha(I_\G)$. We specify a $3 \times 3$ block decomposition on $M$ by declaring that the $(1,1)$ block is $2l \times 2l$, the $(2,2)$ block is $(m-l) \times (m-l)$ and the $(3,3)$ block is $(k-1) \times (k-1)$. First, we show that $M$ is block-diagonal with respect to this block decomposition.

Let $M=(M_{ij})$ where $M_{ij}$ is the matrix in the $(i,j)$ block of $M$. Let \[ D = \mathrm{Diag}(D_1,D_2,D_3) \in \alpha(I_\G), \] as discussed prior to the statement of the proposition. Since $DM = MD$, it must be the case that $M_{32} D_2 = \pm M_{32}$, and $\pm M_{23} = D_2 M_{23}$ for any choice of diagonal matrix $D_2$. This forces $M_{23}$ and $M_{32}$ to be the zero matrix. We must also have $M_{21}D_1 = D_2M_{21}$ and $M_{12} D_2 = D_1 M_{12}$ for any choice of $D_2$ and $D_1$. Taking $D_1$ to be the identity matrix and choosing $D_2$ appropriately allows us to conclude that $M_{21}$ and $M_{12}$ are both the zero matrix. Finally, a similar argument forces $M_{13}$ and $M_{31}$ to also be the zero matrix. Thus, $M = \mathrm{Diag}(M_{11}, M_{22}, M_{33})$.

Since $M \in \mathcal{C}$ and $D$ are both block-diagonal, to determine $\mathcal{C}$ it is necessary and sufficient to centralize within the three diagonal blocks of $D$. For the second and third diagonal blocks, these centralizers are the diagonal subgroup of $\mathrm{GL}_{m-l}(\Z)$ and all of $\mathrm{GL}_{k-1}(\Z)$, respectively. The only task that remains is to determine the centralizer in $\mathrm{GL}_{2l}(\Z)$ of the subgroup $\mathcal{L}$ defined previously.

Suppose that $N \in \mathrm{GL}_{2l}(\Z)$ lies in $C(\mathcal{L})$, the centralizer of the subgroup $\mathcal{L}$. Endow $N$ with a block decomposition compatible with that used to define $\mathcal{L}$: let $N$ have an $l \times l$ block decomposition, where each block is of size $2 \times 2$, writing $N = (N_{ij})$, where $N_{ij}$ is the matrix in the $(i,j)$ block of $N$. Carrying out block matrix multiplication, we see that for $N$ to centralise $\mathcal{L}$ it is necessary that $N_{ii}$ ($1 \leq i \leq l$) commutes with each member of the order 2 subgroup $\mathcal{P}:= \left \langle \begin{pmatrix} 1 & 1 \\ 0 & -1 \end{pmatrix} \right \rangle$, and that $N_{ij} S = TN_{ij}$ ($1 \leq i \neq j \leq l$) for all $S, T \in \mathcal{P}$.

Let $K = \begin{pmatrix} a & b \\ c & d \end{pmatrix}$ appear in some $2 \times 2$ block of $N$. If $K$ lies in a diagonal block, then by the above discussion, we necessarily have
\begin{align*} a &= a + c, \\ a - b &= b+d, \\ c &= -c, \\ c - d &= -d. \end{align*} If $K$ lies off the diagonal of $N$, then its entries must further satisfy the relations \begin{align*} b &= b +d, \\ d &= -d. \end{align*}

In summary, we conclude that in the above block decomposition of $N$, the $i$th diagonal blocks are of the form $\begin{pmatrix} 2b_i + d_i & b_i \\ 0 & d_i \end{pmatrix}$ for some $b_i, d_i \in \Z$, and the off-diagonal $(i,j)$ block is of the form $\begin{pmatrix} 2e_{ij} & e_{ij} \\ 0 & 0 \end{pmatrix}$ for some $e_{ij} \in \Z$. Notice that the even-numbered rows each have precisely one non-zero entry, and so each $d_i$ must lie in $ \{ \pm1 \}$. We have necessary conditions on the entries of the centralizing matrix $N$; we now use these even-numbered rows to obtain sufficient conditions.

To calculate the determinant of $N$, we may consider the determinants of the minors obtained by expanding along these even-numbered rows. Since $\det N = \pm 1$, this expansion allows us to conclude that the matrix
\[ N' := \begin{pmatrix} 
2b_1 + d_1 & 2e_{12} & \cdots & 2e_{1l} \\ 2e_{21} & 2b_2 + d_2 & \cdots &  2e_{2l} \\
\vdots & \vdots & \ddots & \vdots \\
2e_{l1} & 2e_{l2} & \cdots & 2b_l + d_l 
\end{pmatrix} \] lies in $\mathrm{GL}_l(\Z)$. Indeed, $N' \in \Lambda_l[2]$. This observation allows us to define a function \linebreak $\theta : \Lambda_l[2] \to C(\mathcal{L})$ by declaring each $d_i$ in $\theta(A)$ to be 1. Moreoever, due to the placement of zeroes in $N \in C(\mathcal{L})$, $\theta$ is a homomorphism, and an injective one at that. The image of $\theta$ is clearly not surjective, as it cannot contain matrices whose $d_i$ entries are $-1$, however the image is finite index in $C(\mathcal{L})$, as we shall see.

Let $\{f_i \}_{i=1}^l$ be a basis for $(\Z/2)^l$. Define $\xi : (\Z / 2)^l \to C(\mathcal{L})$ with $\xi (f_i)$ being the diagonal matrix in $C(\mathcal{L})$ with $d_i = -1$. We claim that \[ \theta \times \xi : \Lambda_l[2] \times (\Z /2)^l \to C(\mathcal{L}) \] is an isomorphism. Verifying this is a straightforward exercise.

Assembling the centralizers of the three diagonal blocks of $M \in \mathcal{C} \leq \mathrm{GL}_{k+m-1}(\Z)$, we thus obtain that $\mathcal{C}$ is isomorphic to the direct sum in the statement of the proposition.

   \end{proof}
   
\textbf{Remark.} Note that in the above proof of Proposition~\ref{cent}, we realized the principal level 2 congruence subgroup $\Lambda_l[2]$ as a finite index subgroup of a centralizer in $\mathrm{GL}_{2l}(\Z)$ of a finite subgroup. We are not aware of this being exhibited elsewhere in the literature, and it may be of independent interest.

Finally, we determine the image of $\mathcal{C} \leq \Aut(\Outraag)$ in $\Out(\Outraag)$ when we take the quotient by $\Inn(\Outraag)$.

\begin{prop} \label{cbar} The image $\bar {\mathcal{C}}$ of $\mathcal{C}$ in $\Out(\Outraag)$ is isomorphic to \[ \left((\Lambda_l[2] \times (\Z / 2)^l ) / \mathcal{L}\right) \times \mathrm{PGL}_{k-1}(\Z) . \]
\end{prop}

\begin{proof}Consider $w \in \Z^{k+m-1} \leq \Outraag$. Direct computation gives that for any \linebreak $\beta := uh \in \Outraag$ (where $u \in \Z^{k+m-1}$ and $h \in I_\G$), we have $\beta w \beta^{-1} = \alpha(h)(w)$. Since any member of $\mathcal{C}$ preserves $\Z^{k+m-1}$ inside $\Outraag$, we see that $\phi \in \mathcal{C}$ is inner in $\Aut(\Outraag)$ if and only if there exists $h \in I_\G$ such that $\phi (\bar x_i) = \alpha(h)(\bar x_i)$ for all $x_i \in V$. Precisely, we have that \[ \mathcal{C} \cap \Inn(\Outraag) \cong \alpha(I_\G) .\] Immediately we see that $\mathcal{C}$ has infinite image in $\Out(\Outraag)$, but we can determine its structure exactly.

Recall that \[ \mathcal{C} \cong (\Lambda_l[2] \times (\Z / 2)^l) \times (\Z / 2)^{m-l} \times \mathrm{GL}_{k-1}(\Z),\]  by Proposition \ref{cent}. Since $\alpha (I_\G) \cong I_\G \cong (\Z / 2)^n$, the $(\Z / 2)^{m-l}$ factor vanishes in $\Out(\Outraag)$, and we have 
\[ \bar {\mathcal{C}} \cong \left((\Lambda_l[2] \times (\Z / 2)^l ) / \mathcal{L}\right) \times \mathrm{PGL}_{k-1}(\Z) . \] \end{proof}

We have thus established Theorem~\ref{inf}. Over the course of our proofs of Propositions \ref{cent} and \ref{cbar}, we also obtained the following corollary.

\begin{corollary}\label{faith}Given any $\Z$-linear (resp. projective $\Z$-linear) group $G$, there exist infinitely many right-angled Artin groups $\raag$ for which $G$ acts faithfully on $\Outraag$ via automorphisms (resp. outer automorphisms).
\end{corollary}

\textbf{Commensurators.}
If one considers abstract commensurators $\text{Comm}(\Outraag)$ instead of $\Out(\Outraag)$, the discrepancy between our examples and free or free abelian groups is not quite so severe.  We can make this observation precise as follows.  

%Indeed, the abstract commensurators may be seen as lying on a sort of continuum as one progresses from free to free abelian groups.  

Let $\Gamma=(V,E)$ be a focused graph and suppose that $k+m\geq 4$, where $k,m$ are as in the statement of Proposition \ref{pct}.  In particular, we have $n=|V|\geq k+m\geq4$.  At one end, it is a result of Farb and Handel \cite{FH07} that for $n\geq 4$, the abstract commensurator $\text{Comm}(\Out(F_n))$ is just equal to $\Out(F_n)$.  For our focused graph $\Gamma$, by Proposition \ref{pct}, $\Outraag$ is virtually free abelian of rank $k+m-1$, hence its abstract commensurator is $\GL_{k+m-1}(\Q)$.  On the other hand, a theorem of Margulis (see \cite{Mar91}, \cite{Z84}) implies that for $k\geq 3$, the abstract commensurator of $\GL_n(\Z)$ is commensurable with $\GL_n(\Q)$.  

\section{Proof of Theorem B \label{austere}} 
First we recall the definition of an austere graph as defined in \cite{F14}.  A finite simplicial graph $\G=(V,E)$ is called \emph{austere} if it has trivial symmetry group, no dominated vertices, and $\G\setminus \st(v)$ is connected for any $v\in V$.  In particular, we have that no vertex is adjacent to every other vertex, and hence the associated RAAG $A_\G$ has trivial center.

Let $\G$ be an austere graph.  By inspecting the Laurence--Servatius generators, this implies that $\Outraag$ consists only of inversions.  In this case, we know that the automorphism group is a semidirect product:\[\Autraag\cong \Inn(\raag)\rtimes \Out(\raag)\cong \raag\rtimes I_\G,\] where $I_\G$ is the group of inversions. It is easy to write down a presentation for $\Autraag$ in terms of the usual presentation for $\raag$. If $V=\{v_1,\ldots,v_n\}$ is the vertex set of $\G$ then 
%\[\Autraag=\left\langle\begin{array}{l|c}\gamma_1,\ldots, \gamma_n &\mathcal{R}_\G\mbox{; } [\iota_j,\iota_k] \mbox{ for all $j,k$; } {\iota_k}^2\mbox{ for all $k$}\\
%\iota_1,\dots , \iota_n & [\gamma_j,\iota_k]\mbox{ for all $j\neq k$; } (\iota_k\gamma_k)^2\mbox{ for all $k$}
%\end{array}\right\rangle.\]
\[\Autraag=\left\langle\begin{array}{l|c} \gamma_i,\iota_j \mbox{ for $1\leq i,j\leq n$} & \widetilde{\mathcal{R}}_\G \end{array}\right\rangle.\]
Here $\gamma_i$ represents conjugation by $v_i$, $\iota_j$ is inversion of $v_j$, and $\widetilde{\mathcal{R}}_\G$ is comprised of the following five types of relation:
\begin{align} [\gamma_i ,\gamma_k ], & & \mbox{if $v_i$ commutes with $v_k$ in $A_\G$}, \\
[\iota_j,\iota_l], & & 1\leq j,l\leq n, \\
(\iota_j)^2, && 1\leq j\leq n,\\
[\gamma_i,\iota_j], && 1\leq i\neq j\leq n,\\
(\gamma_i\iota_i)^2, && 1\leq i \leq n.
\end{align}

 Using this presentation, we are now ready to prove Theorem \ref{aust}.

%\begin{theorem} Let $\raag$ be an austere graph which is not a single vertex.  Then \[ |\Out(\Autraag)|\geq 2^n.\] 
%\end{theorem}

\emph{Proof of Theorem \ref{aust}.} Let $\pi_k:I_\G=(\Z/2)^n\rightarrow \Z/2$ be the projection onto the $k^{th}$ factor. Consider a function $\phi:\{1,\ldots,n\}\rightarrow I_\G$ satisfying the following two properties:
\begin{enumerate}[(i)]
\item $\pi_k(\phi(k))=0$ for $1\leq k\leq n$.
\item $\pi_j(\phi(k))=0$ whenever $v_k$, $v_j$ commute in $\raag$.   
\end{enumerate}
Then $\phi$ induces a map $\Phi:\Autraag\rightarrow\Autraag$ defined by \[\Phi(\gamma_k)=\gamma_k\cdot\phi(k),\mbox{ }\Phi(\iota_k)=\iota_k.\] Observe that $\Phi$ is an involution, hence a bijection. To see that $\Phi$ is moreover an automorphism, we simply check that is preserves relations.  Since $\Phi$ is the identity on $I_\G$, it is clear that the relations of the form (8) and (9) hold.  Moreover, as all inversions commute with each other, it is clear that the relations of the form (10) are preserved by $\Phi$.  Finally, condition (i) ensures that the relations of type (11) are satisfied, and condition (ii) ensures that the relations of the form (7) still hold.  

The automorphism $\Phi$ is not inner since it does not preserve the subgroup $\Inn(\raag)$, and by composing $\Phi$ with the projection onto $I_\G$, we see that each distinct $\phi$ constructed above gives a distinct $\Phi \in \Out(\Aut(A_\G))$.  Finally, the hypothesis that $\G$ is austere and not a single vertex implies that $\G$ has diameter at least 2.  Thus, given any two non-adjacent vertices it is possible to define a non-zero function $\phi$ as above, hence every vertex contributes at least one such $\Phi$. \qed

\textbf{Remark.} One expects that as $n$ gets large the number of maps $\phi$ which satisfy (i) and (ii) will grow more quickly than exponential in $n$, and hence the automorphism group will be very large.  Indeed, if the maximal valence of any vertex of $\G$ is $k$, then there are at least $2^{n(n-k-1)}$ such $\phi$.    

\bibliographystyle{plain}
\bibliography{outoutbib}
\end{document}